\documentclass[11pt]{amsart}
\usepackage{preamble}

\usepackage{tikz-cd}
\newcommand{\rank}{\operatorname{rank}}

\newcommand{\Aut}{\operatorname{Aut}}

\newcommand{\Ex}{\operatorname{Ex}}

\newcommand{\nklt}{\operatorname{nklt}}

\begin{document}
\title[On quasi-Albanese morphisms for log canonical Calabi-Yau pairs]{On quasi-Albanese morphisms for log canonical Calabi-Yau pairs}


\begin{abstract}
In this note, we study quasi-Albanese morphisms for log canonical Calabi-Yau pairs and obtain several structural results. As an application, we prove a characterization of toric pairs.
\end{abstract}

\author{Yiming Zhu}
\address{School of Mathematical Science, University of Science and Technology of China, Hefei 230026, P.R.China.} \email{zym18119675797@gmail.com}

\maketitle

\setcounter{tocdepth}{1}
\tableofcontents

\section{Introduction}

Let $X$ be a smooth complex projective Calabi-Yau variety. By the Bogomolov-Beauville-Yau decomposition, the Albanese morphism of $X$ is isotrivial. Kawamata \cite[Theorem 8.3]{kawamata1985minimal} extended this isotriviality to varieties with canonical singularities, and Ambro \cite[Theorem 4.8]{Ambro2005moduli} to klt pairs. For log canonical Calabi-Yau pairs, Bernasconi, Filipazzi, Patakfalvi and Tsakanikas \cite[Theorem 1.5]{bernasconi2024strong} proved that the Albanese morphism is locally stable in the sense of Koll\'ar \cite[Definition-Theorem 4.7]{Kollar-Families}, but may not be isotrivial. In the setting of log pairs, it is more natural to consider logarithmic differential forms and quasi-Albanese morphisms by Iitaka \cite{iitaka1976logarithmic}, which capture finer structural information than the Albanese morphisms. 

In this note, we study quasi-Albanese morphisms for log canonical Calabi-Yau pairs. Our results are intended as a step toward a log canonical Bogomolov–Beauville–Yau decomposition.

We briefly recall the relevant definitions. A quasi-abelian variety $G$ is a quasi-projective commutative algebraic group fitting into an exact sequence \[0\to(\Cc^*)^d\to G\to A\to0,\] 
where $A$ is an abelian variety and $(\Cc^*)^d$ is an algebraic torus; we refer to $A$ as the compact part of $G$. Let $X$ be a quasi-projective variety. The quasi-Albanese morphism $\alpha:X\to G$ is a morphism to a quasi-abelian variety characterized by the following universal property: any morphism from $X$ to a quasi-abelian variety factors uniquely through $\alpha$. The variety $G$ is called the quasi-Albanese variety of $X$.  

We recall the following classical result by Kawamata.

\begin{thm}[{\cite[Theorem 28]{kawamata1981characterization}, see also \cite[Theorem 1.2]{fujino2024quasi}}]\label{thm:Kaw}
Let $(X,D_X)$ be a projective log smooth pair with $D_X$ reduced and $\kappa(K_X+D_X)=0$. Then the quasi-Albanese morphism $\alpha:X\setminus D_X\to G$ is dominant and has irreducible general fibers.
\end{thm}

We now turn to the singular setting. For a log canonical pair $(X,D_X)$, we denote by $\nklt(X,D_X)$ its non-klt locus. We strengthen Kawamata' Theorem as follows.

\begin{thm}\label{thm:KawFujSing}\
Let $(X,D_X)$ be a projective log canonical pair with $\kappa(K_X+D_X)=0$. Then the quasi-Albanese morphism $\alpha:X\setminus\nklt(X,D_X)\to G$ is dominant and has irreducible general fibers.
\end{thm}

The next theorem further analyzes the quasi-Albanese morphism $\alpha$ in the Calabi-Yau case.

\begin{thm}\label{thm1}
Let $(X,D_X)$ be a projective log canonical pair with $K_X+D_X\sim_{\Qq}0$. Let $\alpha:X\setminus\nklt(X,D_X)\to G$ be the quasi-Albanese morphism. Then the following holds.
\begin{enumerate}
    \item $\alpha$ is surjective in codimension one: for any prime divisor $P\subset G$, there is a prime divisor $Q\subset X^0$ dominating $P$.
    \item $\alpha:(X\setminus\nklt(X,D_X),D_X|_{X\setminus\nklt(X,D_X)})\to G$ is locally stable. In particular, $\alpha$ is flat and every fiber is reduced with semi-log canonical singularities.
\end{enumerate} 
\end{thm}

Let $G$ be a quasi-abelian variety and let $A$ be its compact part. Then there are $d:=\dim G-\dim A$ line bundles $L_i\in \text{Pic}^0(A)$ such that $G$ is the fiber product over $A$ of the total spaces of the $L_i$ with their zero sections removed. Consequently, $G$ admits the following natural compactification:\[
G\subset\Pp_A:=\Pp_A(O_A\oplus L_1^{-1})\times_A\cdots\times_A \Pp_A(O_A\oplus L_d^{-1}).  \] Note that $H:=\Pp_A\setminus G$ is a simple normal crossing divisor, and the log smooth pair $(\Pp_A,H) $ is Calabi-Yau: $K_{\Pp_A}+H\sim0$. We refer to $(\Pp_A,H)$ as the canonical compactification of $G$.

Let $(X,D_X)$ be a log canonical pair. Via the quasi-Albanese morphism $\alpha:X\setminus\nklt(X,D_X)\to G$, we aim to 
extract structural information
about the pair $(X,D_X)$ itself, rather than merely about the open subset $X\setminus\nklt(X,D_X)$. To this end, we consider the canonical compactification $(\Pp_A,H)$ of $G$ and the induced rational map $X\dashrightarrow\Pp_A$. In general, however, this rational map need not be a morphism (see Example \ref{example}). The following theorem asserts that for a Calabi-Yau pair $(X,D_X)$, one can always find another Calabi-Yau pair $(W,D_W)$, crepant birational to $(X,D_X)$, such that the induced map $W\dashrightarrow\Pp_A$ is a morphism.

\begin{thm}\label{thm2}
Let $(X,D_X)$ be a projective log canonical pair with $K_X+D_X\sim_{\Qq}0$. Let $\alpha:X\setminus\nklt(X,D_X)\to G$ be the quasi-Albanese morphism. Let $(\Pp_A,H)$ be the canonical compactification of $G$. Then there are 
\begin{itemize}
\item a projective $\Qq$-factorial dlt pair $(W,D_W)$ 
with $K_W+D_W\sim_{\Qq}0$, and 
\item a crepant birational contraction $\mu:(W,D_W)\dashrightarrow (X,D_X)$ such that $D_W=\mu^{-1}_*D_X+\sum P$, where the sum runs over all prime divisors on $W$ that are exceptional over $X$,

\end{itemize} satisfying the following properties:

\begin{enumerate}
\item The induced rational map $g:W\dashrightarrow \Pp_A$ is a morphism, which is surjective and has connected fibers.

\item The canonical bundle formula for $g:(W,D_W)\to \Pp_A$ takes the form \[K_W+D_W\sim_{\Qq}g^*(K_{\Pp_A}+H),\] and the restriction $(W,D_W)|_{g^{-1}(G)}\to G$ is locally stable.

\item The support of $g^{*}H$ is contained in the non-klt locus $\nklt(W,D_W)$ and the induced morphism \[\beta:W\setminus\nklt(W,D_W)\to G\] is the quasi-Albanese morphism of $W\setminus \nklt(W,D_W)$.

\end{enumerate}

\end{thm}

As a byproduct of our study, we find a connection between the Shokurov complexity of log canonical pairs and their quasi-Albanese varieties. For a log canonical pair $(X,D_X)$, its complexity $c(X,D_X)$ is defined as \[c(X,D_X):=\dim X+\rank\la\lf D_X\rf\ra-|\lf D_X\rf|,\]
where \begin{itemize}
    \item $\la\lf D_X\rf\ra$ is the subgroup generated by the components of $\lf D_X\rf$ in the group of Weil divisors modulo algebraic equivalence, and
    \item $|\lf D_X\rf|$ is the sum of the coefficients of $\lf D_X\rf$.
\end{itemize} For simplicity, we adopt a definition slightly different from that in \cite{Brown2018GeomToric}. In that work, Brown, McKernan, Svaldi, and Zong proved a conjecture of Shokurov \cite{Shokurov2000CompleSurf} characterizing toric varieties via complexities. See also \cite{Prokhorov2001Toric,yao2013thesis,Nakayama2017Shokurov,enwright2024log} for related results.

In this note, we show the following.

\begin{lem}\label{lem:Comp&Alb}
Let $(X,D_X)$ be a projective log canonical pair. Let $G$ be the quasi-Albanese variety of $X\setminus\nklt(X,D_X)$, and let $A$ be its compact part. Then \[c(X,D_X)=\dim X-(\dim G-\dim A).\]    
\end{lem}

\begin{thm}\label{cor1}
Let $(X,D_X)$ be a projective log canonical pair with $\kappa(K_X+D_X)=0$. Then, for any reduced divisor $0\leq \Delta_X\leq\lf D_X\rf$, we have \[c(X,\Delta_X)\geq0.\] 
Assume that the equality holds for some $\Delta_X$. Then $X$ is rational. Assume further that $K_X+D_X\sim_{\Qq}0$, then the following holds.
\begin{enumerate}
    \item $\Delta_X=\lf D_X\rf=D_X$.
    \item The quasi-Albanese morphism \[\alpha:X\setminus \nklt(X,D_X)\to(\Cc^*)^n\] is birational and isomorphic in codimension one.
    \item There is a crepant birational map \[(X,D_X)\dashrightarrow(\Pp^1_{z_1}\times\cdots\times\Pp^1_{z_n},\sum_i((z_i=0)+(z_i=\infty))).\]
    \item The pair $(X,D_X)$ is toric.
\end{enumerate}  
\end{thm}

A result analogous to Theorem~\ref{thm:Kaw} or Theorem~\ref{thm:KawFujSing} was obtained by Cadorel, Deng, and Yamanoi \cite{Deng2024HyperFund} for Campana’s special manifolds. Recall that projective manifolds with nef anticanonical bundle are special (cf. \cite[Example 2.3]{Campana2004Orbifolds}). 
Moreover, Shokurov’s conjecture \cite{Shokurov2000CompleSurf} only assumes the nefness of $-(K_X+D_X)$. It is therefore natural to expect that the main Theorems of this note remain valid for log canonical pairs with $-(K_X+D_X)$ nef.

The note is organized as follows. In Section 2, we recall the relevant definitions and develop the necessary technical tools. In Section 3, we prove the main theorems.

\subsection*{Acknowledgments}
The author thanks professors Zhan Li, Lei Zhang, Junchao Shentu, and Osamu Fujino for helpful discussions or comments. The author thanks professor Ya Deng for drawing his attention to related results in joint work with Benoit Cadorel and Katsutoshi Yamanoi concerning hyperbolicity and fundamental groups for quasi-projective varieties. The author is supported by USTC.
\section{Preparations}

We work over the field of complex numbers. For the singularities of pairs and the minimal model program, we refer to \cite{KM98,Kollar2013sing}. 

A dominant morphism between normal varieties with irreducible general fibers is called a fibration. Let $f:X\to Y$ be a projective fibration between normal varieties. Let $D$ be an effective $\Qq$-divisor on $X$. We can write $D=D^h+D^v$, where $D^v$ does not dominate $Y$, and each component of $D^h$ dominates $Y$. $D^v$ (resp. $D^h$) is called the vertical (resp. horizontal) part of $D$. $D$ is called vertical (resp. horizontal) if $D=D^v$ (resp. $D=D^h$). Assume that $D$ is vertical, then $D$ is called exceptional over $Y$ if $\codim f(D)\geq2$, and degenerate over $Y$ if for any prime divisor $P\subset Y$, there is a prime divisor $Q\subset X$ such that $Q$ dominates $P$ but $Q\not\subset D$.

For a birational map $\pi:X\dashrightarrow Y$ between normal varieties, we set $\Ex(\pi)=\sum P$, where the sum is taken over all prime divisors on $X$ which are exceptional over $Y$. A birational map $\pi:X\dashrightarrow Y$ is called a birational contraction if $\Ex (\pi^{-1})=0$.

We will use the following fact implicitly: Let $A$, $B$ be two effective $\Qq$-Cartier $\Qq$-divisors on a projective variety $X$. If $\Supp A=\Supp B$, then $\kappa(A)=\kappa(B)$. If $\Supp A\subset \Supp B$, then $\kappa(A)\leq\kappa(B)$. Let $f:Y\to X$ be a projective surjective morphism from another projective variety, then $f^{-1}\Supp A=\Supp f^*A$.

For a projective variety $X$, we set $q(X)=\dim H^1(X,O_X)$, and for a projective lc pair $(X,D_X)$, we set $q(X,D_X)=\dim H^0(X^0, \Omega_{X^0}^1(\log \lf D_X|_{X^0}\rf))$, where $X^0$ is the log smooth locus of $(X,D_X)$.

\subsection{Differential forms on lc pairs}

\begin{thm}[\cite{greb2011differential}]\label{thm:extension} Let $(X,D_X)$ be an lc pair. Let $X^{0}$ be its log smooth locus and let $\nklt(X,D_X)$ be its non-klt locus. Let $\pi:Y\to X$ be a log resolution of $(X,D_X)$ and let $D_Y$ be the codimension one part of $\pi^{-1}\nklt(X,D_X)$. Then $\pi_*\Omega_Y^k(\log  D_Y )$ is reflexive and $H^0(Y,\Omega_Y^k(\log  D_Y ))\cong H^0(X^0, \Omega_{X^0}^k(\log \lf D_X|_{X^0}\rf))$ for all $k$.
\end{thm}

We will use the following Lemma repeatedly.

\begin{lem}\label{lem:compare}
Let $(X,D_X)$ be a projective lc pair. Let $\pi:Y\to X$ be a log resolution of $(X,D_X)$ and let $D_Y=\pi^{-1}_*D_X+\Ex(\pi)$. Then the following holds.
\begin{enumerate}
    \item $q(X,D_X)=q(Y,D_Y)$. 
    \item $\kappa(K_X+D_X)=\kappa(K_Y+D_Y)$ and $\nu(K_X+D_X)=\nu(K_Y+D_Y)$, where $\kappa$ and $\nu$ are Kodaira dimension and numerical Kodaira dimension respectively.
    \item If $(X,D_X)$ is dlt, then $q(X)=q(Y)$.
\end{enumerate} 
\end{lem}
\begin{proof}

Since $(X,D_X)$ is lc, there is an effective exceptional divisor $E$ on $Y$ such that $K_Y+D_Y=\pi^*(K_X+D_X)+E$, so we have (2). If $(X,D_X)$ is dlt, then $X$ has rational singularities by \cite[Theorem 5.22]{KM98}. Hence we have (3).

Let $D_Y'$ be the codimension one part of $\pi^{-1}\nklt(X,D_X)$. Note that $D_Y'\leq \lf D_Y\rf$. To show (1), it suffices to show that the injective homomorphism \[
\iota:H^0(Y,\Omega_Y^1(\log D_Y'))\to H^0(Y,\Omega_Y^1(\log \lf D_Y\rf))    
\] is an isomorphism. This isomorphism is a special case of \cite[Proposition 19.1]{greb2011differential}. Here we give this isomorphism a proof using Hacon's $\Qq$-factorial dlt modification \cite[Proposition 3.3.1]{Hacon2014ACC}.

Step 1. Assume that $(X,D_X)$ is $\Qq$-factorial dlt. Then $\nklt(X,D_X)=\lf D_X\rf$ and $D_{Y}'=q^{-1}\lf D_X\rf$. Write $ D_Y'=\sum_iP_i$ and $\lf D_Y\rf-D_Y'=\sum_jE_j$ as sums of prime divisors. By the exact sequence \[0\to\Omega_Y^1\to\Omega_Y^1(\log (\sum_iP_i+\sum_jE_j))\to \oplus_iO_{P_j}\oplus\oplus_jO_{E_j}\to 0,\]
one checks that $\iota$ is an isomorphism iff $\sum_ia_ic_1(O_Y(P_i))+\sum_jb_jc_1(O_Y(E_j))=0$ implies that all $b_j$'s are zero (see Lemma \ref{lem:ComputeOfq(X,D)}). This is indeed the case since all $E_j$'s are exceptional and $X$ is $\Qq$-factorial.

Step 2. Assume that $(X,D_X)$ is log smooth. By Iitaka \cite[Proposition 1]{iitaka1977logarithmicKodaira}, we have $q(X,\lf D_X\rf)=q(Y,\pi^{-1}\lf D_X\rf)$. Hence we have $q(X,\lf D_X\rf)=q(Y,\lf D_Y\rf)$ by Step 1, i.e., (1) holds for log smooth pairs.

Step 3. In general, consider the following diagram.

\begin{center}
\begin{tikzcd}
W \arrow[r, "q"] \arrow[d, "p"']      & Z \arrow[d, "\mu"] \\
Y \arrow[r, "\pi"] \arrow[ru, dashed] & X                 
\end{tikzcd}    
\end{center}
where 
\begin{itemize}
    \item $(Z,D_Z)$ is a good minimal model of $(Y,D_Y)$ over $X$, hence a Hacon's $\Qq$-factorial dlt modification of $(X,D_X)$ \cite[Proposition 3.3.1]{hacon2007shokurov}, and 
    \item $Y\xleftarrow{p}W\xrightarrow{q}Z$ is a common log resolution.
\end{itemize} 
Since $K_Z+D_Z=\mu^*(K_X+D_X)$, we have $\mu(\lf D_Z\rf)=\nklt$. Hence $\lf D_Z\rf\subset \mu^{-1}\nklt$ and $q^{-1}\lf D_Z\rf\subset (\mu\circ q)^{-1}\nklt$. Let $D_{W}'$ be the codimension one part of $(\mu\circ q)^{-1}\nklt$, then $q^{-1}\lf D_Z\rf\leq D_W'$. Let $D_W=q^{-1}_*D_Z+\Ex(q)$. Then $D_W=(\mu\circ q)^{-1}_*D_X+\Ex(\mu\circ q)$. By $q^{-1}\lf D_Z\rf\leq D_W'\leq \lf D_W\rf$ we have $q(W,q^{-1} \lf D_Z \rf)\leq q(W,D_W')\leq q(W, \lf D_W\rf)$. By Step 1, we have $q(W,q^{-1} \lf D_Z\rf)= q(W, \lf D_W\rf)$. Hence \begin{align}\label{equ1}
q(W,D_W')= q(W,\lf D_W\rf).    
\end{align}

By construction, we have $p^{-1}D_Y'\subset D_W'\subset p^{-1}_*D_{Y}'+\Ex(p)$. Hence $q(W,p^{-1}D_{Y}')\leq q(W,D_W')\leq q(W,p^{-1}_*D_{Y}'+\Ex(p))$. By step 1, we have $q(W,p^{-1}D_{Y}')=q(W,p^{-1}_*D_{Y}'+\Ex(p))$. Hence $q(W,D_W')= q(W,p^{-1}D_Y')$. By Iitaka \cite[Proposition 1]{iitaka1977logarithmicKodaira}, we have \begin{align}\label{equ2}
q(W,D_W')=q(W,p^{-1}D_Y')=q(Y,D_Y')   \end{align} By Step 2, we have \begin{align}\label{equ3}
q(W,\lf D_W\rf)=q(Y,\lf D_Y\rf)    
\end{align} Hence $q(Y,D_Y')=q(Y,\lf D_Y\rf)$ by the three equations \ref{equ1}, \ref{equ2} and \ref{equ3}.
\end{proof}

\begin{lem}\label{lem:compare2}
Let $(X,D_X)$ and $(Y,D_Y)$ be two projective lc pairs. Assume that there is a birational contraction $\pi:X\dashrightarrow Y$ such that $D_X=\pi^{-1}_*D_Y+\Ex(\pi)$ then $q(X,D_X)=q(Y,D_Y)$.
\end{lem}
\begin{proof}
Let $X\xleftarrow{p}W\xrightarrow{q} Y$ be a common log resolution. Then $q^{-1}_*D_Y+\Ex(q)=p^{-1}_*D_X+\Ex(p)$. Hence, the Lemma follows from Lemma \ref{lem:compare}.
\end{proof}

\subsection{Quasi-Albanese morphisms for smooth varieties}

\begin{thm}[\cite{iitaka1976logarithmic}, see also \cite{fujino2024quasi}]\label{thm:Iikata'quasiAlb}
Let $X$ be a smooth quasi-projective variety. Then there is a morphism $\alpha:X\to G$ to a quasi-abelian variety such that 

(1) for any morphism $\alpha':X\to G'$ to a quasi-abelian variety, there is a morphism $h:G\to G'$ such that $\alpha'=h\circ \alpha$, and

(2) $\alpha$ is uniquely determined.
\end{thm}

We recall the constructions of quasi-Albanese morphisms in \cite{iitaka1976logarithmic,fujino2024quasi}. Let $(X,D_X=\sum_{i=1}^mD_i)$ be a projective log smooth pair with $D_X$ reduced. Where $D_i$'s are prime divisors. Let $X^0=X\setminus D_X$ and let $\iota:X^0\to X$ be the open embedding. Let $p:X\to A$ be the Albanese morphism. Set $q=q(X)$, $\overline{q}=q(X,D_X)$, and $d=\overline{q}-q$.

Let $\eta_i$, $1\leq i\leq d$, be a basis of the free part of $ \ker(\iota_*:H_1(X^0,\Zz)\to H_1(X,\Zz))$, and let $\xi_i\in H_1(X^0,\Zz)$, $1\leq i\leq 2q$, be such that the $\eta_i$ and $\xi_i$ form a basis of the free part of $H_1(X^0,\Zz)$. 

Let $\omega_i$, $1\leq i\leq q$, be a basis of $H^0(\Omega_X^1)$, and let $\phi_i\in H^0(\Omega_X^1(\log D))$, $1\leq i\leq d$, be such that the $\omega_i$ and $\phi_i$ form a basis of $H^0(\Omega_X^1(\log D_X))$, and satisfy $\int_{\eta_i}\omega_j=0$ and $\int_{\eta_i}\phi_j=\delta_{ij}$ (the Kronecker delta).

Set
\[A_i=(\int_{\xi_i}\omega_1,\cdots,\int_{\xi_i}\omega_q,\int_{\xi_i}\phi_1,\cdots,\int_{\xi_i}\phi_d)=:(\hat{A}_i,C_i), 1\leq i\leq 2q,\text{ and }\]

\[B_i=(\int_{\eta_i}\omega_1,\cdots,\int_{\eta_i}\omega_q,\int_{\eta_i}\phi_1,\cdots,\int_{\eta_i}\phi_d)=:(0,\hat{B}_i), 1\leq i\leq d.\]
The quasi-Albanese variety of $X^0$ is given by the quasi-abelian variety \[G:=\frac{H^0(X,\Omega_X^1(\log D))^{\vee}}{H_1(X^0,\Zz)}=\frac{\Cc^{\overline{q}}}{\sum_i\Zz A_i+\sum_j\Zz B_j}.\] We have the following Chevalley decomposition of $G$:
\[0\to {\Cc^d}/{\sum_j\Zz \hat{B}_j}\cong(\Cc^*)^d\to G\to {\Cc^q}/{\sum_i\Zz \hat{A}_i}\to0, \]
where ${\Cc^q}/{\sum_i\Zz \hat{A}_i}\cong A$, the Albanese variety of $X$, and the isomorphism $ \Cc^d/\sum_{j}\Zz\hat{B}_j\cong(\Cc^*)^d$ is given by \[(x_1,\cdots,x_d)\mapsto(\exp(2\pi\ii x_1),\cdots,(\exp(2\pi\ii x_d)).\] Fix a point $x_0\in X^0$. The quasi-Albanese morphism of $X^0$ is then given by 
\[X^0\to G, x\mapsto\left(\int_{x_0}^x\omega_1,\cdots,\int_{x_0}^x\omega_q,\int_{x_0}^x\phi_1,\cdots,\int_{x_0}^x\phi_d\right).\]  

Since $G$ is quasi-abelian, there are $d=\dim G-\dim A$ line bundles $L_j\in\text{Pic}^0(A)$ such that $G$ is the fiber product over $A$ of the total spaces of the $L_j$ with their zero sections removed (cf. \cite[Section 2]{Vojta1996Integral}). The following content in this subsection aims to explore these line bundles $L_j$. The explicit form of $L_j$, together with Lemma \ref{lem:ComputeOfq(X,D)}, helps us calculate examples.

Let $G'\to\Cc^q$ be the pullback of the $(\Cc^*)^d$-bundle $G\to A$ by the cover $\Cc^q\to A$. Then we have $G'=\Cc^{\overline{q}}/\sum_i\Zz B_i\cong \Cc^q\times (\Cc^*)^d$ and 
\[G={\Cc^q\times (\Cc^*)^d}/{\pi_1(A)},\] where $\pi_1(A)=\sum_i\Zz{\hat{A}_i}$ acts on $\Cc^q\times (\Cc^*)^d$ diagonally by the group homomorphism: \[\pi_1(A)\to\Aut((\Cc^*)^d),  \hat{A}_i\mapsto\left((t_1,\cdots,t_d)\mapsto(t_1\exp(2\pi \ii\int_{\xi_i}\phi_1),\cdots,t_d \exp(2\pi \ii\int_{\xi_i}\phi_d))\right),1\leq i\leq q.\] 
Hence $L_j$ is the line bundle corresponding to the representation 
\begin{align}\label{equ:repL_i}
\pi_1(A)\to\Cc^*, \hat{A}_i\mapsto \exp(2\pi \ii\int_{\xi_i}\phi_j),1\leq i\leq q.    
\end{align}

Let $a_{ij}=2\pi \ii\mathrm{Res}_{D_i}\phi_j$, $1\leq i\leq m$, $1\leq j\leq d$. For each $j$, we have the following.

\begin{enumerate}
\item $a_{ij}\in\Zz$ for all $i$.
\item $p^*L_j$ is given by the representation $\pi_1(X)\to\Cc^*,\gamma\mapsto \exp(2\pi \ii\int_{\gamma}\phi_j)$.
\item $t_j(x):=\exp(2\pi \ii\int_{x_0}^x\phi_j)$ is a rational section of $p^*L_j$ with $\mathrm{div}(t_j)=\sum_ia_{ij}D_i$. In particular, $p^*L_j\cong O_X(\sum_ia_{ij}D_i)$.
\end{enumerate} 

\begin{proof}
For each $i$, let $\delta_i$ be a small loop around $D_i$. 

(1) Since $\delta_i\in\ker(H_1(X^0,\Zz)\to H_1(X,\Zz))$, we have $\delta_i=\sum_kn_k\eta_k+t$ for some integers $n_k$ and a torsion element $t$. By the residue theorem, we have \[2\pi \ii\mathrm{Res}_{D_i}\phi_j=\int_{\delta_i}\phi_j=\sum_kn_k\int_{\eta_k}\phi_j+\int_t\phi_j=\sum_kn_k\delta_{kj}=n_j\in\Zz.\]

(2)  By the expression $\exp(2\pi \ii\int_{\gamma}\phi_j)$, we mean $\exp(2\pi \ii\int_{\gamma'}\phi_j)$, where $\gamma'$ is any other loop in $X^0$ which is homotopic to $\gamma$. Since $\pi_1(X^0)\to\pi_1(X)$ is surjective with kernel generated by $\delta_i$'s, $\exp(2\pi \ii\int_{\delta_i}\phi_j)=1$ for all $i$, and $\phi_j$ is $d$-closed, the expression $\exp(2\pi \ii\int_{\gamma}\phi_j)$ is well-defined. Hence, we have (2) by the representation (\ref{equ:repL_i}).

(3)  Since $t_j(\gamma x)=\exp(2\pi \ii\int_{\gamma}\phi_j)t_j(x)$ for all $\gamma\in\pi_1(X)$, we have $t_j$ is a holomorphic nowhere vanishing section of $p^*L_j$ over $X^0$. Let $\{x_i\}$ be a coordinate system on an analytic open subset $U$ of $X$ such that $D=(x_1x_2\cdots x_c=0)$. Then \[\phi_j|_U=\sum_{i=1}^c\frac{a_{ij}}{2\pi \ii}\frac{dx_i}{x_i}+\text{some holomorphic function},\text{ and }\]
\[t_j|_U=\prod_{i=1}^cx_i^{a_{ij}} \times\text{some holomorphic function without zeros} .\] Hence $t_j$ is a meromorphic, and so a rational (by Serre's GAGA) section of $p^*L_j$ and $\mathrm{div}(t_j)=\sum_ia_{ij}D_i$.
\end{proof}

Let $(\Pp_A,H)$ be the canonical compactification of $G$. Let $(z_1,\cdots,z_d)$ be the inhomogeneous fiber coordinates of the $(\Pp^1)^d$-bundle $\Pp_A\to A$, the rational map $X\dashrightarrow \Pp_A$ over $A$ is now given by \[x\mapsto (z_1,\cdots,z_d)=(t_1,\cdots,t_d)\in\Pp_A.\]

\begin{exa}\label{example}
Let $(X,D_X)$ be a log smooth pair with $D_X$ reduced.

(1) Let $X=\Pp^2$. If $D_X=$ 3 lines, then $q(X,D_X)=3$, and $X\setminus D_X=(\Cc^*)^2$ itself is a quasi-abelian variety. If $D_X=$ an elliptic curve, then $q(X,D_X)=0$. If $D_X=Q+L$, where $L=(l=0)$ is a line and $Q=(q=0)$ is a conic, then $q(X,D_X)=1$, $H^0(\Omega^1_{X}(\log D_X))=\Cc\cdot d\log\frac{l^2}{q}$ and the quasi-Albanese morphism is $X\setminus D\to\Cc^*, x\mapsto \frac{l^2}{q}$. The rational map $X\dashrightarrow\Pp^1$, $x\mapsto [l^2:q]$ is not defined at $Q\cap L$.

(2) Let $X=\Pp^1_{z_1}\times\Pp^1_{z_2}$ and $D_X=F_1+F_2+\Delta$ where $F_1=(z_1=0)$, $F_2=(z_2=0)$, $\Delta=(z_1=z_2)$. Then $H^0(\Omega_X^1(\log D_X))=\Cc\cdot d\log(1/z_1-1/z_2)$ and the quasi-Albanese morphism is $X\setminus D\to\Cc^*, (z_1,z_2)\mapsto 1/z_1-1/z_2$. The rational map $X\dashrightarrow\Pp^1$, $(z_1,z_2)\mapsto 1/z_1-1/z_2$ is not defined at $(F_1+F_2)\cap\Delta=(z_1=z_2=0)$.

(3) Let $A$ be an abelian variety and let $L$ be a line bundle on $A$. Let $X=\Pp_A(O_A\oplus L)$, let $z$ be its inhomogeneous fiber coordinate, and let $H_0=(z=0)$, $H_{\infty}=(z=\infty)$, $D_X=H_0+H_{\infty}$. Then $q(X)=\dim A$.
 
If $L\in\text{Pic}^0(A)$, then $q(X,D_X)=\dim A+1$, and $\Pp_A\setminus H_1+H_2$ itself is a quasi-abelian variety. 

If $L\not\in\text{Pic}^0(A)$, then $q(X,D_X)=\dim A$, and the quasi-Albanese morphism is the natural projection $\Pp_A\setminus H_1+H_2\to A$.

\end{exa}

\begin{lem}\label{lem:DX<DX'}
Let $X$ be a smooth projective variety. Let $D_X\leq D_X'$ be reduced snc divisors. By the universal property for quasi-Albanese morphisms, we have the following commutative diagram.

\begin{center}
\begin{tikzcd}
X\setminus D_X' \arrow[r] \arrow[d, hook] & G' \arrow[d] \\
X\setminus D_X \arrow[r]  & G           
\end{tikzcd} 
\end{center}
where the horizontal arrows are quasi-Albanese morphisms and $G'\to G$ is the induced morphism. Assume that $q(X,D_X)=q(X,D_X')$. Then $G'\to G$ is an isomorphism.
\end{lem}
\begin{proof}
Since the injective homomorphism \[H^0(X,\Omega_X^1(\log D_X))\to H^0(X,\Omega_X^1(\log D_X'))\] is an isomorphism, we have $H^1(X\setminus D_X,\Cc)\to H^1(X\setminus D_X',\Cc)$ is an isomorphism by Deligne's mixed Hodge structures:
\begin{center}
\begin{tikzcd}
{H^1(X\setminus D_X,\Cc)} \arrow[r] \arrow[d, "\cong"]    & {H^1(X\setminus D_X',\Cc)} \arrow[d, "\cong"] \\
H^1(X,O_X)\oplus H^0(X,\Omega_X^1(\log D_X)) \arrow[r, "\cong"] & H^1(X,O_X)\oplus H^0(X,\Omega_X^1(\log D_X'))    
\end{tikzcd}    
\end{center}
Hence \[G'=\frac{H^0(X,\Omega_X^1(\log D_X'))^{\vee}}{H_1(Z\setminus D_X',\Zz)}\to G=\frac{H^0(X,\Omega_X^1(\log D_X))^{\vee}}{H_1(Z\setminus D_X,\Zz)}\] is an isomorphism. 
\end{proof}

\subsection{Quasi-Albanese morphisms for klt varieties}\label{rem:Quasi-albanSing}

Let $X$ be a (possibly) singular variety. Let $\pi:X'\to X$ be a resolution and let $f':X'\to G$ be the quasi-Albanese morphism of $X'$, then the quasi-Albanese map of $X$ is defined as the rational map $f'\circ \pi^{-1}$. Assume that $X$ has klt singularities. Then the fibers of $\pi$ are rationally chain connected by Hacon-McKernan \cite{hacon2007shokurov}. Since $G$ contains no rational curves, $f'$ factors through $\pi$ by the rigidity Lemma (cf. \cite[5.3 Proposition]{KollarRationalCurves}). In other words, there is a morphism $f:X\to G$ such that $f\circ \pi=f'$. 
\begin{center}
\begin{tikzcd}
X' \arrow[r, "f'"] \arrow[d, "\pi"'] & G \\
X \arrow[ru, "\exists f"']           &  
\end{tikzcd}    
\end{center}
One checks that the morphism $f:X\to G$ given in this way satisfies the universal property as in Theorem \ref{thm:Iikata'quasiAlb}.

Let $(X,D_X)$ be a projective lc pair and let $\nklt=\nklt(X,D_X)$ be its non-klt locus. Let $\alpha:X\setminus\nklt\to G$ be the quasi-Albanese morphism. Let $A$ be the compact part of $G$ and let $A':=(\text{Pic}^0(X))^{\vee}_{\text{red}}$ be the Albanese variety of $X$. Let $\pi:Z\to X$ be a log resolution such that $\pi^{-1}\nklt$ is a divisor. By our definition, $G$ is the quasi-Albanese variety of $Z\setminus\pi^{-1}\nklt$. Hence $A$ is the Albanese variety of $Z$ by Iitaka's construction. By the universal property for Albanese morphisms, we have the following commutative diagram.

\begin{center}
\begin{tikzcd}
Z \arrow[d] \arrow[r] & A \arrow[d] \\
X \arrow[r]           & A'         
\end{tikzcd}     
\end{center}
where $A\to A'$ is the morphism induced by $Z\to X$. Unlike the log smooth case, in general, we have $A\not\cong A'$. But if $(X,D_X)$ is dlt, then by \cite{hacon2007shokurov} and the rigidity Lemma, we do have $A\cong A'$.

\begin{lem}\label{lem:LcToLogSmooth}Let $(X,D_X)$ be a projective lc pair. Let $G$ be the quasi-Albanese variety of $X\setminus\nklt(X,D_X)$. Let $\pi:Z\to X$ be a log resolution such that $D_Z':=\pi^{-1}\nklt(X,D_X)$ is a divisor. Let $D_Z=\pi^{-1}_*D_X+\Ex(\pi)$. Note that $D_Z'\leq \lf D_Z\rf$. By the universal property for quasi-Albanese morphisms, we have the following commutative diagram.

\begin{center}
\begin{tikzcd}
Z\setminus \lf D_Z\rf \arrow[r] \arrow[d,hook] & G' \arrow[d] \\
Z\setminus D_Z' \arrow[r]  & G           
\end{tikzcd} 
\end{center}
where the horizontal arrows are quasi-Albanese morphisms and $G'\to G$ is the induced morphism. Then $G'\to G$ is an isomorphism. In particular, we have $\dim G=q(X,D_X)$.
\end{lem}

\begin{proof}
By Lemma \ref{lem:compare}, we have $q(X,D_X)=q(Z,D_Z')=q(Z,D_Z)$. Hence the Lemma follows from Lemma \ref{lem:DX<DX'}.
\end{proof}

\subsection{Subadditivity and positivity}

The following is a summary of slightly generalizations of the relevant theorems due to Kawamata \cite{kawamata1981characterization}, Maehara \cite{maehara1986weak}, Campana \cite{Campana2004Orbifolds} and Fujino \cite{Fujino2017notes}, and it is one of the most important techniques used in this note. 

\begin{thm}\label{thm:subadditivity}
Let $f:X\to Y$ be a fibration between smooth projective varieties. Let $D_X$ and $D_Y$ be  divisors on $X$ and $Y$ respectively, such that $(X,D_X)$ and $(Y,D_Y)$ are log smooth lc. Assume that $D_Y$ is reduced. Denote by $F$ a very general fiber of $f$. Assume that \[f^{-1}D_Y\subset \lf D_X\rf.\] Then the following holds.

(1)(cf. \cite[Theorem 30]{kawamata1981characterization}, \cite[Corollary 1, Corollary 2]{maehara1986weak}\label{thm:maehara}, \cite[Theorem 4.2]{Campana2004Orbifolds} or \cite[Theorem 1.9]{Fujino2017notes}) Let $B$ be a big $\Qq$-divisor on $Y$, then \[\kappa(K_{X}+D_X-f^*(K_Y+D_Y)+f^*B)=\kappa(K_F+{D_X}|_F)+\dim Y.\] 
In particular, if $K_Y+D_Y$ is big, then \[\kappa(K_X+D_X)=\kappa(K_F+D_X|_F)+\dim Y.\]

(2) If $\kappa(K_F+{D_X}|_F)\geq0$, then $K_{X}+D_X-f^*(K_Y+D_Y)$ is pseudo-effective.

(3) (\cite{fujino2017subadditivity,Fujino2017Corrigendum})$\nu(K_X+D_X)\geq\nu(K_F+D_X|_F)+\kappa(K_Y+D_Y)$. 

\end{thm}

\begin{proof}
Note that we do not require $D_X$ to be reduced here. Using \cite[Theorem 1.1]{Fujino2017notes}, the proof of (1) is identical to that of \cite[Theorem 1.9]{Fujino2017notes}. Since the $\Qq$-divisor $B$ in (1) can be sufficiently small, (2) follows.    
\end{proof}

Let $(X^0,D^0)$ be a quasi-projective log smooth lc pair, and let $X^0\subset X$ be a compactification such that $(X,D+B)$ is log smooth, where $B=X\setminus X^0$ and $D$ is the closure of $D^0$ in $X$. We define the logarithmic Kodaira dimension $\overline{\kappa}$ (resp. numerical logarithmic Kodaira dimension $\overline{\nu}$) of $(X^0,D^0)$ as $\overline{\kappa}(X^0,D^0):=\kappa(K_{X}+D+B)$ (resp. $\overline{\nu}(X^0,D^0):=\nu(K_{X}+D+B)$).  To show that $\overline{\kappa}$ and $\overline{\nu}$ are well defined, it suffices to show that any two such compactifications give the same $\overline{\kappa}$ and $\overline{\nu}$. Let $(X,D+B)$ and $(X',D'+B')$ be two such compactifications. Let $X\xleftarrow{p}X''\xrightarrow{q}X'$ be a common resolution, which is an isomorphism over $X^0$, and such that $D''+B''$ is snc, where $B''=X''\setminus X^0$ and $D''$ is the closure of $D^0$ in $X''$. Then $D''=p^{-1}_*D$ and $B''=p^{-1}_*B+\Ex(p)$. Hence \[K_{X''}+D''+B''=p^*(K_X+D+B)+\sum_{E\subset \Ex(p)}(a(E,X,D+B)+1)E.\] 
Since $(X,D+B)$ is lc, we have $\sum_{E\subset \Ex(p)}(a(E,X,D+B)+1)E\geq0$. Hence $\kappa(K_X+D+B)=\kappa(K_{X''}+D''+B'')=\kappa(K_{X'}+D'+B')$ and $\nu(K_X+D+B)=\nu(K_{X''}+D''+B'')=\nu(K_{X'}+D'+B')$. Using these definitions, we can restate Theorem \ref{thm:subadditivity} (1) (3) as follows.

\begin{lem}\label{lem:subadditivity2}
Let $f^0:X^0\to Y^0$ be a fibration between smooth quasi-projective varieties. Let $D^0$ be a divisor on $X^0$ such that $(X^0,D^0)$ is log smooth lc. Denote by $F^0$ a very general fiber of $f^0$. Then the following holds.\begin{enumerate}
\item If $\overline{\kappa}(Y^0)=\dim Y^0$, then $\overline{\kappa}(X^0,D^0)= \overline{\kappa}(F^0,D^0|_{F^0})+\dim Y^0$.
     
\item $\overline{\nu}(X^0,D^0)\geq \overline{\nu}(F^0,D^0|_{F^0})+\overline{\kappa}(Y^0)$.  
 
\end{enumerate} 
\end{lem}
\begin{proof}
Let $X^0\subset X$ and $Y^0\subset Y$ be compactifications, such that $(X,D+B_X)$, $(Y,B_Y)$ are log smooth and there is a morphism $f:X\to Y$ extending $f^0$, where $D$ is the closure of $D^0$ in $X$, $B_X=X\setminus X^0$ and $B_Y=Y\setminus Y^0$. By construction, we have $f^{-1}B_Y\subset B_X\subset \lf D+B_X\rf$. Then (1), (2) follow from Theorem \ref{thm:subadditivity} (1), (3) respectively.
\end{proof}

\section{Proof of the main theorems}

\subsection{Proof of theorem \ref{thm:KawFujSing}}

\begin{proof}[Proof of theorem \ref{thm:KawFujSing}] 
This is (3) of the following Theorem \ref{lem:KawFujLogSmooth}. \end{proof}

\begin{thm}\label{lem:KawFujLogSmooth}
\begin{enumerate}
\item Let $(X^0,D_X^0)$ be a quasi-projective log smooth lc pair with $\overline{\kappa}(X^0,D_X^0)=0$. Then the quasi-Albanese morphism of $X^0$ is a fibration. 

\item Let $(X,D_X)$ be a projective $\Qq$-factorial dlt pair with $\kappa(K_X+D_X)=0$. Let $0\leq\Delta_X\leq \lf D_X\rf$ be a reduced divisor. Then the quasi-Albanese morphism of $X\setminus \Delta_X$ is a fibration.

\item Let $(X,D_X)$ be a projective lc pair with $\kappa(K_X+D_X)=0$. Then the quasi-Albanese morphism of $X\setminus \nklt(X,D_X)$ is a fibration.

\end{enumerate}
\end{thm}
\begin{proof}
(1) The proof is identical to that of \cite[Theorem 1.2]{fujino2024quasi} by using Lemma \ref{lem:subadditivity2} (1) instead of \cite[Theorem 6.1]{fujino2024quasi}.

(2) Let $\pi:Z\to X$ be a log resolution and let $D_Z=\pi^{-1}_*D_X+\Ex(\pi)$. Since  $\pi^{-1}\Delta_X\leq \lf D_Z\rf$, we have $\overline{\kappa}(Z\setminus \pi^{-1}\Delta_X,D_Z-\pi^{-1}\Delta_X)=\kappa(K_Z+D_Z)=\kappa(K_X+D_X)=0$. By (1), the quasi-Albanese morphism of $Z\setminus \pi^{-1}\Delta_X$ is a fibration. Hence we have (2).

(3) Let $\pi:Z\to X$ be a log resolution such that $\pi^{-1}\nklt(X,D_X)$ is a divisor. Let $D_Z=\pi^{-1}_*D_X+\Ex(\pi)$. Since  $\pi^{-1}\nklt\leq \lf D_Z\rf$, we have $\overline{\kappa}(Z\setminus \pi^{-1}\nklt,D_Z-\pi^{-1}\nklt)=\kappa(K_Z+D_Z)=\kappa(K_X+D_X)=0$. By (1), the quasi-Albanese morphism of $Z\setminus \pi^{-1}\nklt$ is a fibration. Hence we have (3).
\end{proof}

\subsection{Proof of theorem \ref{thm2}}

In the following theorem, we do not assume that $G$ is the quasi-Albanese variety of $Z\setminus\lf D_Z\rf$ and we do not assume that $G\neq A$ (in particular, $(\Pp_A,H)$ may coincides with $(A,0)$). 

\begin{thm}\label{thm:key3}
Let $(Z,D_Z)$ be a projective log smooth lc pair with $\kappa(K_Z+D_Z)=0$. 
Let $G$ be a quasi-abelian variety and let $A$ be its compact part. Let $(\Pp_A,H)$ be the canonical compactification of $G$. Note that $(\Pp_A,H)$ is log smooth and $K_{\Pp_A}+H\sim0$.

Assume that there are
\begin{itemize}
\item a fibration $f:Z\to \Pp_A$ with 
 $f^{-1}H\subset \lf D_Z\rf$, 
 \item a projective lc pair $(X,D_X)$ with $K_X+D_X\sim_{\Qq}0$, and
 \item a birational morphism $\pi:Z\to X$ such that $D_Z=\pi^{-1}_*D_X+\Ex(\pi)$. Note that since $(X,D_X)$ is lc, we have \[K_Z+D_Z=\pi^*(K_X+D_X)+E_Z\] for some effective and exceptional over $X$ divisor $E_Z$.
 \end{itemize} 

Then the following holds.

\begin{enumerate}

\item\label{(1)}
Let $Q\subset Z$ be a prime divisor which is exceptional over $\Pp_A$. Then $Q\subset E_Z+f^{-1}H$.

\item\label{VeryExceoOverG} Let $P\subset\Pp_A$ be a prime divisor with $P\not\subset H$, then there is a prime divisor $Q\subset Z$ with $Q\not\subset E_Z$ dominating $P$.

\item\label{AssumeVeryExcepOverH} Assume that for any prime divisor $P\subset H$, there is a prime divisor $Q\not\subset E_Z$ dominating $P$, then
$(Z,D_Z)$ has a good minimal model $(W,D_W)$ over $\Pp_A$ with $K_W+D_W\sim_{\Qq}0$.

\item\label{assumeMMP} Assume that $(Z,D_Z)$ has a good minimal model $(W,D_W)$ over $\Pp_A$ with $K_W+D_W\sim_{\Qq}0$. Let $g:W\to\Pp_A$ be the induced fibration. Then the canonical bundle formula for $(W,D_W)\to\Pp_A$ has the form \[K_W+D_W\sim_{\Qq}g^*(K_{\Pp_A}+H),\] the restriction $(W,D_W)|_{g^{-1}G}\to G$ is locally stable, and $D_W^v=\lf D_W^v\rf=g^{-1}H$. The induced birational map $\mu:(W,D_W)\dashrightarrow (X,D_X)$ is a crepant birational contraction with $D_W=\mu^{-1}_*D_X+\Ex(\mu)$.

\item\label{indeedMMP} $(Z,D_Z)$ has a good minimal model $(W,D_W)$ over $\Pp_A$ with $K_W+D_W\sim_{\Qq}0$.
\end{enumerate}
\end{thm}
\begin{proof}
Let $F$ be a very general fiber of $f$. By Fujino's subadditivity Theorem \ref{thm:subadditivity}(3), we have \[0=\nu(K_Z+D_Z)\geq\nu(K_F+D_Z|_F)+\kappa(K_{\Pp_A}+H).\] Hence $\nu(K_F+D_Z|_F)=0$. By Gongyo's abundance theorem \cite{gongyo2011minimal}, we have $\kappa(K_F+D_Z|_F)=0$ and $(F,D_Z|_F)$ has a good minimal model.

(1) Assume that $Q\not\subset f^{-1}H$. Let $q:\Pp'_A\to\Pp_A$ be a composition of the blowup along $f(Q)$ with a log resolution. We may assume that $q$ is an isomorphism over $\Pp_A\setminus f(Q)$. Let $p:Z'\to Z$ be a log resolution of $(Z,D_Z)$ resolving the indeterminacies of $Z\dashrightarrow \Pp_A'$. We have the following commutative diagram.

\begin{center}
\begin{tikzcd}
Z' \arrow[d, "f'"] \arrow[r, "p"] & Z \arrow[d, "f"] \arrow[r, "\pi"] & X \\
\Pp_A' \arrow[r, "q"]                 & \Pp_A                                 &  
\end{tikzcd}    
\end{center}

Write $K_{Z'}+D_{Z'}=p^*(K_Z+D_Z)+E_1$,  $K_{\Pp'_A}+H'=q^*(K_{\Pp_A}+H)+E_2$, where $H'=q^{-1}H$, $D_{Z'}=p^{-1}_*D_Z+\Ex(p)$, and $E_1$, $E_2$ are effective and exceptional divisors. By $f^{-1}H\subset \lf D_Z\rf$, we have $f'^{-1}H'\subset \lf D_{Z'}\rf$. By \[0\sim_{\Qq}p^*(K_{Z}+D_Z-f^*(K_{\Pp_A}+H)-E_Z)=K_{Z'}+D_{Z'}-f'^*(K_{\Pp'_A}+H')-E_1+f'^*E_2-p^*E_Z,\] $\kappa(K_F+D_Z|_F)=0$, and Theorem \ref{thm:subadditivity}(2), we have $E_1-f'^*E_2+p^*E_Z$ is pseudo-effective. Since $(\pi\circ p)_*(E_1-f'^*E_2+p^*E_Z)=-(\pi\circ p)_*f'^*E_2$ is pseudo-effective, we have $f'^*E_2$ is exceptional over $X$. Hence $E_1-f'^*E_2+p^*E_Z$ is effective by a Lemma due to Lazarsfeld (cf. \cite[Corollary 13]{kollar2010quotients}).

Let $Q'$ be the strict transform of $Q$ on $Z'$ and let $P'=f'(Q')$. Since $q(P')=f(Q)\not\subset H$, we have $\mult_{P'}E_2>0$, so $\mult_{Q'}f'^*E_2>0$. By \[0\leq\mult_{Q'}(E_1-f'^*E_2+p^*E_Z)=\mult_{Q'}(-f'^*E_2+p^*E_Z),\] we have $Q'\subset p^*E_Z$. Hence $Q\subset E_Z$.

(2) Assume the contrary, i.e., assume that \begin{align*}
f^{-1}P\subset E_Z+E    
\end{align*}
for some effective exceptional over $\Pp_A$ divisor $E$. Then by (\ref{(1)}), we have \[f^{-1}(P+H)\subset E_Z+f^{-1}H\subset \lf D_Z\rf.\] Consider the fibration $Z\setminus\lf D_Z\rf\to \Pp_A\setminus P+H$. By Lemma \ref{lem:subadditivity2}(2) and Lemma \ref{lem:G'sKappa}, we have \[0=\nu(K_Z+D_Z)=\overline{\nu}(Z\setminus \lf D_Z\rf,D_Z)\geq \overline{\nu}(F\setminus \lf D_Z\rf, D_Z)+\overline{\kappa}({\Pp_A}\setminus P+H)=\overline{\kappa}(G\setminus P)>0,\] which is a contradiction.

(3) By (\ref{VeryExceoOverG}) and the assumption in (\ref{AssumeVeryExcepOverH}), we known that the vertical part $E_Z^v$ of $E_Z$ is degenerate over $\Pp_A$. Since $\kappa(E_Z|_F)=\kappa(K_F+D_Z|_F)=0$, we have \begin{align}\label{equ:reflexiv}
O_{\Pp_A}\subset f_*O_Z(mE_Z)\subset (f_*O_Z(mE_Z))^{\wedge}=O_{\Pp_A}
\end{align} for all sufficiently divisible $m>0$, where $\wedge$ means taking reflexive hull. In particular, \[\bigoplus_{m\geq0}f_*O_Z(m(K_Z+D_Z))=\bigoplus_{m\geq0}f_*O_Z(mE_Z)\] is a finitely generated $O_{\Pp_A}$-algebra. By Gongyo's abundance theorem and Hacon-Xu \cite[Theorem 2.12]{hacon2013existence}, $(Z,D_Z)$ has a good minimal model $(W,D_W)$ over $\Pp_A$. Denote by $g:W\to\Pp_A$ the induced fibration. Let $m>0$ be a sufficiently divisible integer. Then there is a surjective morphism \[g^*g_*O_{W}(m(K_{W}+D_W))\to O_W(m(K_{W}+D_W)).\] By formula (\ref{equ:reflexiv}), we have \[g_*O_{W}(m(K_{W}+D_W))=f_*O_{Z}(m(K_{Z}+D_Z))=O_{\Pp_A}.\] Hence  $K_W+D_W\sim_{\Qq}0$.

(4) Write the canonical bundle formula for $g:(W,D_W)\to \Pp_A$ as $K_W+D_W\sim_{\Qq}g^*(K_{\Pp_A}+B_{\Pp_A}+M_{\Pp_A})$. Since $f^{-1}H\subset \lf D_Z\rf$, we have $g^{-1}H\subset \lf D_W\rf$. Hence $B_{\Pp_A}\geq H$. Since $K_W+D_W\sim_{\Qq}0$ and $K_{\Pp_A}+H\sim0$, we have $B_{\Pp_A}=H$ and $M_{\Pp_A}\sim_{\Qq}0$. Hence the canonical bundle formula writes \begin{align}\label{equ:CBF2}
K_{W}+D_W\sim_{\Qq}g^*(K_{\Pp_A}+H).    
\end{align}
Let $\Sigma$ be a reduced snc divisor on $G$ and let $\overline{\Sigma}$ be its closure in $\Pp_A$. Then the canonical bundle formula for $(W,D_W+g^*\overline{\Sigma})\to \Pp_A$ writes $K_W+D_W+g^*\overline{\Sigma}\sim_{\Qq}g^*(K_{\Pp_A}+H+\overline{\Sigma})$. By inversion of adjunction (cf. \cite[Theorem 3.1]{Ambro2004Shokurov} or \cite[Proposition 4.16]{filipazzi2018generalized}), $(W,D_W+g^*\overline{\Sigma})|_{g^{-1}G}$ is lc. Hence the restriction $(W,D_W)|_{g^{-1}G}\to G$ is locally stable by \cite[Corollary 4.55]{Kollar-Families}. In particular, by \cite[Definition-Theorem 4.7]{Kollar-Families}, the restriction $g|_{g^{-1}G}$ is flat.

Let $Q\subset D_W^v$ be a prime divisor. If $g(Q)\not\subset H$, then $P:=g(Q)$ is a divisor on $\Pp_A$. Since $(W,D_W+g^*P)$ is lc near the generic point of $P$, we have \[-1\leq a(Q,W,D_W+g^*P)=-\mult_Q(D_W+g^*P)=-1-\mult_QD_W<-1,\] which is a contradiction. Hence $D_W^v\subset g^{-1}H\subset \lf D_W^v\rf$.

Since the birational map $Z\dashrightarrow W$ contracts precisely the divisor $E_Z\subset\mathrm{Ex}(\pi)$, we have 
$W\dashrightarrow X$ is a birational contraction and $D_W=\mu^{-1}_*D_X+\Ex(\mu)$. Let $X\xleftarrow{u}V\xrightarrow{v}W$ be a common log resolution of $(X,D_X)$ and $(W,D_W)$. Then $u^*(K_X+D_X)=v^*(K_W+D_W)$ by the Negativity Lemma (cf. \cite[3.39]{KM98}), i.e., $\mu:(W,D_W)\dashrightarrow (X,D_X)$ is crepant.

(5) Case 1. $\dim\Pp_A-\dim A=0$. By Gongyo's abundance theorem, $(Z,D_Z)$ has a good minimal model $(W,D_W)$. Let $Z\xleftarrow{u}V\xrightarrow{v}W$ be a common resolution. By \cite{hacon2007shokurov}, the rigidity Lemma, and the universal property for Albanese morphisms, we have the following commutative diagram.
\begin{center}
\begin{tikzcd}
V \arrow[r, "u"] \arrow[rd, "v"'] & Z \arrow[d, dashed] \arrow[r, "p_Z"] & A_Z \arrow[d, "\cong"'] \arrow[r] & A \\
& W \arrow[r, "p_W"]   & A_W \arrow[ru, "\exists"']         &  
\end{tikzcd}    
\end{center}
where $p_Z$, $p_W$, $p_Z\circ u$, $p_W\circ v$ are Albanese morphisms, $A_Z\to A$ is the induced morphism by the universal property of $p_Z$. Hence we may identify the Albanese varieties of $W$ and $Z$, and there is a morphism $W\to A$, compatible with the rational map $Z\dashrightarrow W$. Hence $(W,D_W)$ is actually a good minimal model of $(Z,D_Z)$ over $A$.

Case 2. $\dim\Pp_A-\dim A>0$. By (\ref{AssumeVeryExcepOverH}),to prove (\ref{indeedMMP}), it suffices to show that for any prime divisor $P\subset H$, there is a prime divisor $Q\not\subset E_Z$ dominating $P$. Assume the contrary, i.e., assume that there is a prime divisor $P\subset H$, and an effective exceptional over $\Pp_A$ divisor $E$ with $f(E)\subset P$ such that \begin{align}\label{equ:very-excep}
f^{-1}P\subset E_Z+E.
\end{align}

Let $L_i\in\text{Pic}^0(A)$, $1\leq i\leq d$, be line bundles such that $G$ is the fiber product over $A$ of the total spaces of $L_i$'s with their zero sections removed. Then
$\Pp_A=\Pp_A(O_A\oplus L_1^{-1})\times_A\cdots\times_A\Pp_A(O_A\oplus L_d^{-1})$. Let $(z_1,\cdots,z_d)$ be the inhomogeneous fiber coordinates of $\Pp_A$. For each $i$, let $p_i:\Pp_A\to \Pp_A(O_A\oplus L_i^{-1})$ be the projection, and let \[H_i=(z_i=0)+(z_i=\infty)\subset \Pp_A(O_A\oplus L_i^{-1})\] be the effective and ample over $A$ divisor. Then $H=\sum_{i=1}^dp_i^*H_i$.

Since $P\subset H=\sum_{i=1}^dp_i^*H_i$, we have $P\subset p_i^*H_i $ for some $i$. By re-indexing $i$'s, we may assume that $P\subset p_1^*H_1$. Let \[\Pp'_A=\Pp_A(O_A\oplus L_2^{-1})\times_A\cdots\times_A\Pp_A(O_A\oplus L_d^{-1}),\] and let $p_i':\Pp_A'\to \Pp_A(O_A\oplus L_i^{-1})$, $2\leq i\leq d$, $\alpha:\Pp_A\to\Pp_A'$, $p':\Pp_A'\to A$ be the natural projections. Let $H'=\sum_{i=2}^d {p'}_i^*H_i$. Then  $\Pp_A'\setminus H'$ is a quasi-abelian variety whose canonical compactification is just $(\Pp_A',H')$, and $\alpha\circ f:Z\to \Pp_A'$ is a fibration with \[(\alpha\circ f)^{-1}H'\subset f^{-1}H\subset \lf D_Z\rf.\] By induction on $(\dim\Pp_A-\dim A)$, $(Z,D_Z)$ has a good minimal model $(V,D_{V})$ over $\Pp_A'$ with $K_{V}+D_{V}\sim_{\Qq}0$. We have the following commutative diagram.

\begin{center}
\begin{tikzcd}
Z \arrow[d, dashed] \arrow[r, "f"] & \Pp_A \arrow[d, "\alpha"] \\
V \arrow[r, "g"]   & \Pp_A' \arrow[d, "p'"]     \\
& A  
\end{tikzcd}
\end{center}

We claim that \begin{align}\label{claim2}
f^{-1}P\subset E_Z+(\alpha\circ f)^{-1}H'.    
\end{align}
Let $Q\subset E$ be a prime divisor which is not contained in $E_Z$. Since $f(E)\subset P\subset H$, we have $Q\subset f^{-1}H\subset \lf D_Z\rf $. Since the MMP $Z\dashrightarrow V$ contracts precisely the divisor $E_Z$, the strict transform $Q_V$ of $Q$ on $V$ is a divisor contained in $\lf D_V\rf$. Since $\dim\Pp_A-\dim\Pp_A'=1$, and $Q$ is exceptional over $\Pp_A$, we have $Q$ is vertical over $\Pp_A'$. Hence $Q_V$ is vertical over $\Pp_A'$. By (\ref{assumeMMP}), we have $Q_V\subset g^{-1}H'$. Hence $Q\subset (\alpha\circ f)^{-1}H'$ and the claim follows.

Let $L_A$ be an effective and sufficiently ample divisor on $A$. Since $H'$ and $\alpha^*H'+P$ are ample over $A$, we have 
\begin{align*}
\dim\Pp_A&=\kappa(P+\alpha^*H'+(p'\circ \alpha)^*L_A)
\\&=\kappa(f^*P+(\alpha\circ f)^*H'+(p'\circ \alpha\circ f)^*L_A)
\\&\leq \kappa(E_Z+(\alpha\circ f)^*(H'+p'^*L_A)) \text{ (by inclusion (\ref{claim2}))}
\\&=\kappa(H'+p'^*L_A) \text{ (by Lemma \ref{lem:covering})}
\\&=\dim\Pp_A',    
\end{align*} which is a contradiction. Hence the inclusion (\ref{equ:very-excep}) can not be true.
\end{proof}\label{lem:G'sKappa}

\begin{lem}[{cf. \cite[Proposition 1.3]{Deng2024HyperFund}}]
Let $G$ be a quasi-abelian variety and let $P\subset G$ be a prime divisor, then $\overline{\kappa}(G\setminus P)>0$. 
\end{lem}
\begin{proof}
The idea is from \cite[Proof of Theorem 1.3, Step 2]{fujino2024quasi}. Let $A$ be the compact part of $G$ and let $p:G\to A$ be the natural projection.

Case 1. $P$ is vertical over $A$. Then $P=p^*Q$ for some prime divisor $Q$ on $A$. Let $(\Pp_A,H)$ be the canonical compactification of $G$. Let $\overline{P}$ be the closure of $P$ in $\Pp_A$. Since any lc center of $(\Pp_A,H)$ is horizontal over $A$, $\overline{P}$ contains no lc centers of $(\Pp_A,H)$. Then $\overline{\kappa}(G\setminus P)=\kappa(K_{\Pp_A}+H+\overline{P})=\kappa(\overline{P})=\kappa(Q)>0$ by Lemma \ref{lem:KappaBar}.
 
Case 2. $P$ is horizontal over $A$. Then there is a subgroup $\Cc^*\subset G$ such that $P$ dominates the quasi-abelian variety $G':=G/\Cc^*$. Let $F$ be a general fiber of $G\setminus P\to G'$. Then $F$ is a curve with $\overline{\kappa}(F)>0$. Hence $\overline{\kappa}(G\setminus P)\geq\overline{\kappa}(F)+\overline{\kappa}(G')=\overline{\kappa}(F)>0$ by Kawamata \cite[Theorem 1]{kawamata1977addition} (see also Fujino \cite[Chapter 5]{fujino2020iitaka}).\end{proof}

\begin{lem}\label{lem:KappaBar}
Let $(Y,D_Y)$ be a projective log smooth  pair with $D_Y$ reduced. Assume that $\kappa(K_Y+D_Y)\geq0$. Let $\Sigma\subset Y$ be a reduced divisor which contains no lc centers of $(Y,D_Y)$, then $
\overline{\kappa}(Y\setminus D_Y+ \Sigma)=\kappa(K_Y+D_Y+\Sigma)$.
\end{lem}
\begin{proof}
Let $\mu:Z\to Y$ be a log resolution of $(Y,D_Y+\Sigma)$. We may assume that $\mu$ is an isomorphism over $Y\setminus D_Y+ \Sigma$. Let $\epsilon>0$ be a sufficiently small rational number. Then
\begin{align*}
\overline{\kappa}(Y\setminus D_Y+ \Sigma)
&
=\overline{\kappa}(Z\setminus\mu^{-1}(D_Y+\Sigma))
\\&=\kappa(K_Z+\mu^{-1}(D_Y+\Sigma))
\\&=\kappa(K_Z+\mu^{-1}_*D_Y+\mu^{-1}_*\Sigma+\Ex(\mu))
\\&=\kappa(K_Z+\mu^{-1}_*D_Y+\epsilon\mu^{-1}_*\Sigma+\Ex(\mu))\text{ (by $\kappa(K_Z+\mu^{-1}_*D_Y+\Ex(\mu))\geq0$)} 
\\&=\kappa(K_Y+D_Y+\epsilon\Sigma) \text{ (by $(Y,D_Y+\epsilon\Sigma)$ is lc)}. 
\\&=\kappa(K_Y+D_Y+\Sigma) \text{ (by $\kappa(K_Y+D_Y)\geq0$)}. 
\end{align*}
\end{proof}

\begin{lem}\label{lem:covering}
Let $\alpha:Z\to V$ and $\beta:W\to V$ be fibrations between projective normal varieties, and let $Z\dashrightarrow W$ be a birational contraction over $V$. Let $E$ be an effective $\Qq$-Cartier $\Qq$-divisor on $Z$ which is  exceptional over $W$, and let $L$ be a $\Qq$-Cartier $\Qq$-divisor on $V$. Then $\kappa(\alpha^*L+E)=\kappa(L)$.
\end{lem}
\begin{proof}
Let $u:X\to Z$ be a birational morphism resolving the indeterminacies of $Z\dashrightarrow W$ and let $v:X\to W$ be the induced morphism. Since $Z\dashrightarrow W$ is a birational contraction, $u^*E$ is exceptional over $W$. Hence \[
\kappa(\alpha^*L+E)=\kappa(u^*\alpha^*L+u^*E)=\kappa(v^*\beta^*L+u^*E)
=\kappa(\beta^*L)
=\kappa(L).\]\end{proof}

\begin{thm}[A stronger dlt version of Theorem \ref{thm2}]\label{thm2dlt}
Let $(X,D_X)$ be a projective $\Qq$-factorial dlt pair with $K_X+D_X\sim_{\Qq}0$. Let $\Delta_X\leq \lf D_X\rf$ be a reduced divisor. Let $\alpha:X\setminus \Delta_X\to G$ be the quasi-Albanese morphism. Let $A$ be the compact part of $G$. Let $(\Pp_A,H)$ be the canonical compactification of $G$. Then there are 

\begin{itemize}
\item a projective $\Qq$-factorial dlt pair $(W,D_W)$ 
with $K_W+D_W\sim_{\Qq}0$ and 
\item a crepant birational contraction $\mu:(W,D_W)\dashrightarrow(X,D_X)$ with $D_W=\mu^{-1}_*D_X+\Ex(\mu)$,
\end{itemize} such that the following holds.
\begin{enumerate}
    \item The rational map $g:W\dashrightarrow \Pp_A$ is a fibration,
    \item The canonical bundle formula for $(W,D_W)\to \Pp_A$ writes $K_W+D_W\sim_{\Qq}g^*(K_{\Pp_A}+H)$, and the restriction $(W,D_W)|_{g^{-1}G}\to G$ is locally stable.
    \item Let $\Delta_W=\mu^{-1}_*\Delta_X+\Ex(\mu)$. Then $\Delta^v_W=D_W^v=\lf D_W^v\rf=g^{-1}H$ and the morphism $W\setminus \Delta_W\to G$ is the quasi-Albanese morphism of $W\setminus \Delta_W$. 
\end{enumerate} 

\end{thm}
\begin{proof}
Let $\pi:Z\to X$ be a log resolution such that the rational map $f:Z\dashrightarrow\Pp_A$ is a morphism. By Theorem \ref{lem:KawFujLogSmooth}(2), $f$ is a fibration. Let $D_Z=\pi^{-1}_*D_X+\Ex(\pi)$ and let $\Delta_Z=\pi^{-1}_*\Delta_X+\Ex(\pi)$.
Then $K_Z+D_Z=\pi^*(K_X+D_X)+E_Z$ for some effective and exceptional divisor $E_Z$ and $f^{-1}H\subset \Delta_Z\subset\lf D_Z\rf$. 

By Theorem \ref{thm:key3}, there is a good minimal model $(W,D_W)$ of $(Z,D_Z)$ over $\Pp_A$ satisfying (1) and (2) and such that the birational map $\mu:(W,D_W)\dashrightarrow(X,D_X)$ is a crepant birational contraction with $D_W=\mu^{-1}_*D_X+\Ex(\mu)$. 

By $\lf D_W^v\rf= g^{-1}H\subset \Delta_W^v\subset\lf D_W^v\rf$, we have $\Delta_W^v=\lf D_W^v\rf$. Let $W\setminus\Delta_W\to G'$ be the quasi-Albanese morphism. Since $W\setminus\Delta_W\to G$ is a fibration, the induced morphism $G'\to G$ is a fibration. By Lemmas \ref{lem:compare2} and \ref{lem:LcToLogSmooth}, we have $\dim G'=q(W,\Delta_W)=q(X,\Delta_X)=\dim G$. Hence the fibration $G'\to G$ is an isomorphism, i.e., $W\setminus\Delta_W\to G$ is the quasi-Albanese morphism.
\end{proof}

\begin{proof}[Proof of Theorem \ref{thm2}]
Let $\pi:Z\to X$ be a log resolution such that $\pi^{-1}\nklt$ is a divisor and the rational map $f:Z\dashrightarrow\Pp_A$ is a morphism. By Theorem \ref{lem:KawFujLogSmooth}(3), $f$ is a fibration. Let $D_Z=\pi^{-1}_*D_X+\Ex(\pi)$. Then $K_Z+D_Z=\pi^*(K_X+D_X)+E_Z$ for some effective and exceptional divisor $E_Z$ and $f^{-1}H\subset \pi^{-1}\nklt\subset\lf D_Z\rf$. 

By Theorem \ref{thm:key3}, there is a good minimal model $(W,D_W)$ of $(Z,D_Z)$ over $\Pp_A$ satisfying (1) and (2) and such that the birational map $\mu:(W,D_W)\dashrightarrow(X,D_X)$ is a crepant birational contraction with $D_W=\mu^{-1}_*D_X+\Ex(\mu)$. 

By the same reasoning as in the previous Theorem \ref{thm2dlt}, we have $W\setminus\lf D_W\rf\to G$ is the quasi-Albanese morphism.
\end{proof}

\begin{proof}[Proof of Theorem \ref{thm1}]
Keep the notations introduced in the proof of Theorem \ref{thm2}. 

(1) Let $P\subset G$ be a prime divisor and let $\overline{P}$ be its closure in $\Pp_A$. Then $\overline{P}\not\subset H$. Since $\lf D_W^v\rf=g^{-1}H$, there is a prime divisor $\overline{Q}_W\subset W$ such that $\overline{Q}_W\not\subset \lf D_W\rf$ and $\overline{Q}_W$ dominates $\overline{P}$. Since $\Ex(\mu)\leq\lf D_W\rf$, $\overline{Q}:=\mu_*\overline{Q}_W$ is a prime divisor on $X$ such that $\overline{Q}\not\subset\lf D_X\rf$ (in particular $\overline{Q}\not\subset\nklt(X,D_X)$) and $\overline{Q}$ dominates $P$.

(2) Let $\Sigma\subset G$ be a reduced snc divisor and let $\overline{\Sigma}$ be its closure in $\Pp_A$. Since $(W,D_W)|_{g^{-1}G}\to G$ is locally stable, we have $(W,D_W+g^*\overline{\Sigma})|_{g^{-1}G}$ is lc. Hence $(Z,D_Z-E_Z+f^*\overline{\Sigma})|_{f^{-1}G}$ is sub-lc. By \[(K_Z+D_Z-E_Z+f^*\overline{\Sigma})|_{Z\setminus \pi^{-1}\nklt}=(\pi|_{Z\setminus \pi^{-1}\nklt})^*(K_{X\setminus\nklt}+D_X|_{X\setminus\nklt}+\alpha^*\Sigma)\] and $Z\setminus \pi^{-1}\nklt\subset Z\setminus f^{-1}H=f^{-1}{G}$, we have $(X\setminus\nklt+D_X|_{X\setminus\nklt},\alpha^*\Sigma)$ is lc. Hence $\alpha:(X\setminus\nklt,D_X|_{X\setminus\nklt})\to G$ is locally stable by \cite[Corollary 4.55]{Kollar-Families}. It is flat with slc fibers by \cite[Definition-Theorem 4.7]{Kollar-Families}
\end{proof}

\subsection{Proofs of Lemma \ref{lem:Comp&Alb} and Theorem \ref{cor1}}
\begin{lem}\label{lem:ComputeOfq(X,D)}
Let $(X,D_X)$ be a projective log smooth pair with $D_X$ reduced. Write $D_X=\sum_{i=1}^mD_i$ as a sum of prime divisors. Then 
\[q(X,D_X)-q(X)=\{(a_1,\cdots,a_m)\in\Cc^m\mid \sum_ia_ic_1(O_X(D_i))=0\text{ in }H^1(\Omega_X^1)\},\text{ and}\] 
\[m-(q(X,D_X)-q(X))=\text{the dimension of the subspace spanned by }c_1(O_X(D_i))\text{'s in }H^1(\Omega_X^1).\] 

\end{lem}
\begin{proof}
Consider the exact sequence \[0\to \Omega_X^1\to\Omega_X^1(\log D_X)\to \bigoplus_{i=1}^mO_{D_i}\to 0,\] 
and the associated long exact sequence:
\[0\to H^0(\Omega_X^1)\to H^0(\Omega_X^1(\log D_X))\to \bigoplus_{i=1}^mH^0(O_{D_i})\xrightarrow{\delta} H^1(\Omega_X^1).\] 
Since $\delta(1|_{D_i})=c_1(O_X(D_i))\in H^1(\Omega_X^1)$ (cf. \cite[Chapter III Exercise 7.4]{HartshorneGTM52}), where $1|_{D_i}$ is the constant function on $D_i$ with value $1$, the Lemma follows.
\end{proof}

\begin{lem}[{cf. \cite[Lemma 2.41]{enwright2024log}}]\label{lem:compare:complx}
Let $(X,D_X)$ be a projective lc pair. Let $\pi:Z\to X$ be a log resolution and let $D_Z=\pi^{-1}_*D_X+\Ex(\pi)$. Then $c(Z,D_Z)=c(X,D_X)$.
\end{lem}
\begin{proof}
Write $\Ex(\pi)=\sum_{i=1}^mE_i$ as a sum of prime divisors. Since $|\lf D_Z\rf|=|\lf D_X\rf|+m$, to prove the Lemma, it suffices to show that $\rank \la\lf D_Z\rf\ra=\rank \la\lf D_X\rf\ra+m$. Since $\pi_*:\la\lf D_Z\rf\ra\otimes\Qq\to \la\lf D_X\rf\ra \otimes\Qq$ is surjective and $E_i$'s form a basis of its kernel (by the negativity Lemma \cite[3.39]{KM98}), the Lemma follows.
\end{proof}

\begin{proof}[Proof of Lemma \ref{lem:Comp&Alb}]
Let $\pi:Z\to X$ be a log resolution such that $\pi^{-1}\nklt(X,D_X)$ is a divisor. Let $D_Z=\pi^{-1}_*D_X+\Ex(\pi)$. By the previous two Lemmas, we have
$c(X,D_X)=c(Z,D_Z)=\dim Z-(q(Z,D_Z)-q(Z))$. By Lemma \ref{lem:LcToLogSmooth}, $X\setminus\nklt(X,D_X)$ and $Z\setminus\lf D_Z\rf$ have the isomorphic quasi-Albanese variety $G$. By Iitaka's construction, we have $\dim G=q(Z,D_Z)$ and $\dim A=q(Z)$. Hence $c(X,D_X)=\dim X-(\dim G-\dim A)$.
\end{proof}

\begin{proof}[Proof of Theorem \ref{cor1}]Let $\Delta_X\leq\lf D_X\rf$ be a reduced divisor. Let $\pi:Z\to X$ be a Hacon's $\Qq$-factorial dlt modification \cite[Proposition 3.3.1]{Hacon2014ACC}. Let $D_Z=\pi^{-1}_*D_X+\Ex(\pi)$ and let $\Delta_Z=\pi^{-1}_*\Delta_X+\Ex(\pi)$. Then $K_Z+D_Z=\pi^*(K_X+D_X)$, $\kappa(K_Z+D_Z)=0$ and $\Delta_Z\leq \lf D_Z\rf$. 

By Lemma \ref{lem:compare:complx}, we have $c(X,\Delta_X)=c(Z,\Delta_Z)$. By Proposition \ref{lem:Comp&Alb}, we have $c(Z,\Delta_Z)=\dim Z-(\dim G-\dim A)$, where $G$ and $A$ are the quasi-Albanese variety and the Albanese variety of $Z\setminus \Delta_Z$ and $Z$ respectively. By \ref{lem:KawFujLogSmooth}(2), we have $\dim Z\geq \dim G$. Hence $c(X,\Delta_X)\geq0$.
 
Assume that $c(X,\Delta_X)=0$. Then $\dim Z=\dim G$ and $\dim A=0$. Hence $G\cong(\Cc^*)^n$. The quasi-Albanese morphism $Z\setminus\Delta_Z\to (\Cc^*)^n$ is birational, so $X$ is rational. Assume further that $K_X+D_X\sim_{\Qq}0$. Then $K_Z+D_Z\sim_{\Qq}0$. Set $((\Pp^1)^n,H):=(\Pp^1_{z_1}\times\cdots\times\Pp^1_{z_n},\sum_i((z_i=0)+(z_i=\infty)))$. By Theorem \ref{thm2dlt}, there are
\begin{itemize}
\item a $\Qq$-factorial dlt pair $(W,D_W)$ with $K_W+D_W\sim_{\Qq}0$,
\item a crepant birational contraction $\mu:(W,D_W)\dashrightarrow (Z,D_Z)$ with $D_W=\mu^{-1}_*D_Z+\Ex(\mu)$, and

\item a birational morphism $g:(W,D_W)\to ((\Pp^1)^n,H)$ with $K_W+D_W\sim_{\Qq}g^*(K_{(\Pp^1)^n}+H)$, and $\Delta_W=\lf D_W\rf=D_W=g^{-1}H$, where $\Delta_W=\mu^{-1}_*\Delta_Z+\Ex(\mu)$.

\end{itemize}
Hence we have $\Delta_Z=\lf D_Z\rf=D_Z$, $\Delta_X=\lf D_X\rf=D_X$ and the four pairs $(X,D_X)$, $(Z,D_Z)$, $(W,D_W)$  $((\Pp^1)^n,H)$ are crepant birational to each other.

Let $P\subset W$ be a prime divisor which is exceptional over $(\Pp^1)^n$, then $P\subset g^{-1}H=D_W$. By $-1=a(P,W,D_W)=a(P,(\Pp^1)^n,H)$, we have $P$ is toric. By \cite[Lemma 2.3.2]{Brown2018GeomToric}, we have $(W,D_W)$, $(Z,D_Z)$ and $(X,D_X)$ are toric. 

By taking a log resolution of $(Z,D_Z)$ and Lemma \ref{lem:LcToLogSmooth}, we known that $Z\setminus\lf D_Z\rf$ and $X\setminus\nklt(X,D_X)$ have the same quasi-Albanese variety $(\Cc^*)^n$. By Theorem \ref{thm1}, we have the quasi-Albanese morphism $X\setminus\nklt(X,D_X)\to (\Cc^*)^n$ is birational and isomorphic in codimension one.
\end{proof}



\bibliographystyle{alpha}

\bibliography{bibfile}

\end{document}